\documentclass[12pt,reqno]{amsart}

\usepackage{a4wide}
\usepackage{amsmath,amsthm,amssymb}
\usepackage{enumerate}
\usepackage{graphicx}

\def\A{\mathcal{A}}

\def\C{\mathcal{C}}

\def\O{\mathcal{O}}
\def\S{\mathcal{S}}

\def\V{\mathbb{V}}

\def\muaup{\overline{\mu}}

\def\musup{\overline{\sigma}}

\def\sn{s\cdot n}
\def\d{\,{\rm d}}
\def\ds{\,{\rm ds}}
\def\dr{\,{\rm dr}}

\def\RR{\mathbb{R}}

\def\D{\R\times\S}
\def\dD{\Gamma}
\def\R{\mathcal{R}}
\def\dR{{\partial \R}}

\def\md{\mathcal{M}}

\newcommand\diam{\mathop{\rm diam}}

\newtheorem{theorem}{Theorem}[section]

\newtheorem{lemma}[theorem]{Lemma}
\newtheorem{corollary}[theorem]{Corollary}

\begin{document}
\title[Numerical identification for radiative transfer]{Numerical methods for parameter identification in stationary radiative transfer}
\date{\today}
\author[H. Egger]{Herbert Egger$^\dag$}
\author[M. Schlottbom]{Matthias Schlottbom$^\dag$}
\thanks{$^\dag$Numerical Analysis and Scientific Computing, Department of Mathematics, TU Darmstadt, Dolivostr. 15, 64293 Darmstadt. \\
Email: {\tt $\{$egger,schlottbom$\}$@mathematik.tu-darmstadt.de}}

\begin{abstract}
We consider the identification of scattering and absorption rates in the stationary radiative transfer equation. For a stable solution of this parameter identification problem, we consider Tikhonov regularization within Banach spaces. A regularized solution is then defined via an optimal control problem constrained by an integro partial differential equation. By establishing the weak-continuity of the parameter-to-solution map, we are able to ensure the existence of minimizers and thus the well-posedness of the regularization method. In addition, we prove certain differentiability properties, which allow us to construct numerical algorithms for finding the minimizers and to analyze their convergence. Numerical results are presented to support the theoretical findings and illustrate the necessity of the assumptions made in the analysis.
\end{abstract}

\maketitle

{\footnotesize

}

{\footnotesize

}

\footnotetext{Department of Mathematics, Numerical Analysis and Scientific Computing, Technische Universit\"at Darmstadt, Dolivostr. 15, D--64293 Darmstadt, Germany.}

\section{Introduction}


We consider the stationary radiative transfer equation 
\begin{align}\label{eq:rte}
s\cdot\nabla \phi(r,s) + \mu(r) \phi(r,s) = \sigma(r) \Big( \int_\S \phi(r,s')\ds' - \phi(r,s)\Big) + f(r,s),
\end{align}
which models the equilibrium distribution of an ensemble of mutually non-interacting particles in an isotropically scattering medium. 
The function $\phi(r,s)$ here denotes the density of particles at a point $r\in\R$ moving in direction 
$s\in\S={\mathbb S}^{d-1}$, and the symbol $\nabla$ denotes derivatives with respect to the spatial variables $r$ only.
The medium is characterized by rates $\mu$ and $\sigma$ of absorption and scattering. 
Interior sources are represented by $f$ and the inflow of particles over the boundary 
is modeled by
\begin{align}\label{eq:rte_bc}
 \phi(r,s)=g(r,s) \qquad\text{for } r\in\dR, s\cdot n(r)<0,
\end{align}
where $n(r)$ is the unit outward normal at a point $r\in\dR$. 
Problems of the form \eqref{eq:rte}--\eqref{eq:rte_bc} arise in various applications, 
e.g., in neutron physics \cite{CaseZweifel67}, in medical imaging \cite{Arridge99}, 
in astrophysics \cite{Cercignani:1988,Chandrasekhar60,Peraiah04}, or climatology \cite{Kondratyev69}.

In this paper, we are interested in the determination of 
the material properties, encoded in the spatially varying parameters $\mu$ and $\sigma$
from measurements
\begin{align}\label{eq:obs}
 B\phi = \int_{s\cdot n(r)>0} \phi(r,s) \sn \ds
\end{align}
of the outflow of particles over the boundary. 
This parameter identification problem can be formally written as an abstract operator equation
\begin{align}\label{eq:ip_informal}
 B S(\mu,\sigma) = \md,
\end{align}
where $\md$ is a given measurement, and $S$ denotes the parameter-to-solution map defined by 
$S(\mu,\sigma)=\phi$ solving \eqref{eq:rte}--\eqref{eq:rte_bc}.
Note that $S$ and $\md$ also depend on the sources $f$ and $g$.

Due to the many important applications, the inverse problem \eqref{eq:ip_informal} has been investigated 
intensively in the literature. To give an impression of the basic properties of the problem, 
let us summarize some of the most important results:
The parameters $\mu$, $\sigma$ can be uniquely identified, if sufficiently many measurements are available
\cite{ChoulliStefanov98}.
In particular, multiple excitations $f$ and $g$ are required. The stability of the identification 
process with respect to perturbations in the data has been investigated in
\cite{Bal08,BalJol08,McDowallStefanovTamason2010}. In general, the stability will be very low.
Various methods to numerically solve the parameter identification problem 
have been proposed as well \cite{Dorn98,KloHie99,WrightArridgeSchweiger07}.

It is by now well understood that solving \eqref{eq:ip_informal} is an ill-posed problem.
For a stable solution, we will therefore consider Tikhonov regularization, to be precise, we
define approximate solutions via minimization problems of the form 
\begin{align}\label{eq:tikhonov}
  &\| BS(\mu,\sigma)- \md\|_{L^q(\dR)}^q  + \alpha \|\mu-\mu_0\|^p_{L^p(\R)} + \alpha \|\sigma-\sigma_0\|^p_{L^p(\R)} 
  \to \min_{(\mu,\sigma) \in D(S)},
\end{align}
where $\mu_0$ and $\sigma_0$ denote some a-priori information about the unknown parameters $\mu$, $\sigma$.
The domain $D(S)$ will be defined below. 
This can be seen as an optimal control problem governed by an integro partial differential equation.

The main focus of this manuscript is to establish the existence of minimizers for \eqref{eq:tikhonov}
and thus to ensure the well-posedness of the regularized problem.
We will also show that \eqref{eq:tikhonov} is a regularization method in the sense of \cite{EHN96}. 
In addition, we will investigate iterative algorithms to approximate the minimizers.  
The key ingredient for our arguments is a careful analysis of the mapping properties of the parameter-to-solution map $S$. We will establish its strong and weak continuity with respect to the corresponding $L^p$ and $L^q$ topologies, 
and derive various differentiability results. 
Let us mention that for particular choices of the parameter and measurement spaces, 
the stable solution of the inverse problem \eqref{eq:ip_informal} by Tikhonov regularization 
has been considered already in \cite{DiDoNaPaSi02,ThesisSch,TangHanHan2013}. 
Our results here are more general and require 
a much finer analysis of the operator $S$.
We will make more detailed comments on this in the following sections.
%
As a numerical method for minimizing the Tikhonov functional, we consider a variation of the iteratively regularized Gau\ss-Newton method. This method has been investigated in the framework of regularization methods in \cite{BakKok04,KaNeuSche08}. Here, we investigate its properties for minimization of the regularized functional.

The outline of the manuscript is as follows: in Section~\ref{sec:prelim}, we introduce the necessary notation and recall some existence results for the transport equation. 
After fixing the domain of $S$, we proof our main results about
continuity, weak continuity, and differentiability of $S$ in Section~\ref{sec:properties}. 
We turn back to the optimal control problem in Section~\ref{sec:OCproblem} and investigate iterative methods for its solution in Section~\ref{sec:iterative}. 
For illustration of our theoretical considerations, some numerical results are presented in Section~\ref{sec:numerics},
and we conclude with a short summary.

\section{Preliminaries} \label{sec:prelim}

Let us introduce the basic notions and the functional analytic setting in which we investigate the 
solvability of the radiative transfer problem.
The following physically reasonable and quite general assumptions will be used throughout the paper.
\begin{enumerate} \label{ass:1} \itemsep1ex
\item[(A1)] $\R\subset\RR^3$ is a bounded domain with Lipschitz boundary.
\item[(A2)] $\mu\in L^\infty(\R)$ and $0\leq \mu(r) \leq \muaup$ for a.e.\@ $r \in \R$ with some constant $\muaup \ge 0$.
\item[(A3)] $\sigma\in L^\infty(\R)$ and $0 \leq \sigma(r) \leq \musup$ for a.e.\@ $r \in \R$ with some constant $\musup \ge 0$.
\end{enumerate}
Since $\dR$ is Lipschitz continuous, we can define for almost every $r \in \dR$ the outward unit normal vector $n=n(r)$. We denote by $\dD:=\dR \times \S$ the boundary of the tensor product domain $\D$ and decompose $\dD$ into an in- and outflow part by
\begin{align}\label{eq:def_dD}
 \dD_{\pm} :=\{(r,s)\in\dR\times \S: \pm s\cdot n(r) >0\}.
\end{align}
We will search for solutions of the radiative transfer problem \eqref{eq:rte}--\eqref{eq:rte_bc} in the space
\begin{align*}
 \V^p &:= \{ v\in L^p(\D): s\cdot \nabla v\in L^p(\D) \text{ and } v\in L^p(\dD_-; |\sn|)\}
\end{align*}
which is equipped with the graph norm
\begin{align*}
 \| v\|^p_{\V^p} &:= \|v\|^p_{L^p(\D)} + \|s\cdot\nabla v\|^p_{L^p(\D)} + \|v\|_{L^p(\dD_-;|\sn|)}^p.
\end{align*}
Here $L^p(\dD_-;|\sn|)$ denotes a weighted $L^p$-space with weighting function $|\sn|$.
Note that for $1 \le p \le \infty$, the spaces $\V^p$ are complete and that $\V^2$ is a Hilbert space. 
Due to the boundedness of the spatial domain $\R$, the embedding $\V^p \hookrightarrow \V^q$ is continuous for $q \le p$, but neither $\V^p \hookrightarrow \V^q$ nor $\V^p \hookrightarrow L^p(\D)$ are compact.
For functions  $u\in \V^p$ and $v \in \V^q$ with $q=1-1/p$, we obtain the integration-by-parts formula
\begin{align} \label{eq:green}
(s \cdot \nabla u, v)_{\D} = -(u, s \cdot \nabla v)_{\D} + (s \cdot n\; u, v)_{\dD}.
\end{align}
As usual, the symbol $(u,v)_D$ is used for the integral of the product of two functions over some domain $D$.
Applying this formula to $u \in \V^p$ and $v = u |u|^{p-2}$ yields
\begin{align} \label{eq:outflow}
\|u\|_{L^p(\dD_+;|s \cdot n|)}^p 
\le \|u\|_{L^p(\dD_;|s \cdot n|)}^p  + p\|u\|_{L^p(\D)} \|s \cdot \nabla u \|_{L^p(\D)},
\end{align}
i.e., the outflow trace of functions in $\V^p$ is well-defined and the trace operator is continuous from $\V^p$ to $L^p(\dD_+;|\sn|)$. 
Via H\"older's inequality, we immediately obtain
\begin{lemma} \label{lem:obs}
The operator $B:\V^p \to L^p(\dD_+;|\sn|)$ defined in \eqref{eq:obs} is linear and bounded.
\end{lemma}
Let us introduce the transport operator
\begin{align*}
 \A:\V^p \to L^p(\D), \quad (\A\phi)(r,s):= s\cdot\nabla\phi(r,s)
\end{align*}
which models the flow of particles in direction $s$, and the averaging operator
\begin{align*}
 \Theta:L^p(\D)   \to    L^p(\D), \quad (\Theta\phi)(r,s):= \int_\S \phi(r,s')\ds',
\end{align*}
describing the scattering of particles by the background medium. 
The collision operator 
\begin{align*}
\C = \mu I + \sigma (I - \Theta)
\end{align*}
then models the total interaction of particles with the medium. 
Note, that $\C$ depends linearly on the parameters $\mu$ and $\sigma$, and we will sometimes write $\C(\mu,\sigma)$
to emphasize this dependence.
For later reference, let us summarize some basic properties of the operators, 
which follow more or less directly from their definition; see \cite{DL93vol6,EggSch10:3} for details.
\begin{lemma}
Let (A1)--(A3) hold. 
Then the operators $\A :\V^p\to L^p(\D)$, $\Theta: L^p(\D)\to L^p(\D)$, and $\C:L^p(\D) \to L^p(\D)$ are bounded linear operators. Moreover, $\Theta$ and $\C$ are self-adjoint and $\C$ is positive on $L^2(\D)$.
\end{lemma}

As already mentioned, the energy spaces $\V^p$ are not compactly embedded in $L^p(\D)$. 
The following result, known as averaging lemma, serves as a substitute and will play a key-role in 
our analysis. 
\begin{lemma}\label{lem:scatteringcompactness}
  For any $1<p<\infty$ the averaging operator $\Theta:\V_0^p\to L^p(\R)$ is compact.
 Here $\V_0^p$ denotes the subspace of $\V^p$ with vanishing inflow boundary conditions.
\end{lemma}
We refer to \cite{GoLiPeSe88} for a proof of this result.
Let us mention that averaging lemmas also play a key role for the spectral analysis of the radiative transfer equation.

Using the operators defined above,  
the radiative transfer problem \eqref{eq:rte}--\eqref{eq:rte_bc} can be written in compact form: 
Given $f \in L^p(\D)$ and $g \in L^p(\dD_-;|\sn|)$, find $\phi \in \V^p$ such that 
\begin{alignat}{3}  \label{eq:op_eq} 
\A \phi + \C \phi &= f &\quad&\text{in } \D, & \\
 \phi &= g &\quad &\text{on } \dD_-. \label{eq:op_bc}
\end{alignat}
The two equations have to hold in the sense of $L^p(\D)$ and $L^p(\dD_-;|s\cdot n|)$, respectively. 
The existence and uniqueness of solutions for this problem is established next. 
\begin{theorem}\label{thm:existence}
Let (A1)--(A3) hold. Then for any $f\in L^p(\D)$ and $g\in L^p(\dD_-; |s\cdot n|)$, $1\leq p \leq\infty$, 
the radiative transfer problem \eqref{eq:op_eq}--\eqref{eq:op_bc} has a unique solution $\phi\in \V^p$ 
and 
 \begin{align*}
  \|\phi\|_{\V^p}\leq C \big(\|f\|_{L^p(\D)} + \|g\|_{L^p(\dD_-;|s\cdot n|)}\big)
 \end{align*}
 with a constant $C$ depending only on $\diam\R$, $p$ and the bounds $\overline{\mu}$ and $\overline{\sigma}$ in (A2)--(A3).
\end{theorem}
For a proof of this and further results, let us refer to \cite{DL93vol6,EggSch13LP} and the references given there.

\section{Properties of the parameter-to-solution map} \label{sec:properties}

In this section, we investigate the mapping properties of the parameter-to-solution map
\begin{align}
S : D(S) \subset L^p(\R) \times L^p(\R) \to \V^p, \qquad (\mu,\sigma) \mapsto \phi,
\end{align}
where $\phi$ is the solution of \eqref{eq:op_eq}--\eqref{eq:op_bc} for given data $f$ and $g$. 
The domain of $S$ is defined by
\begin{align*}
 D(S):=\{ (\mu,\sigma)\in L^p(\R)\times L^p(\R): \text{(A2)--(A3) hold} \}.
\end{align*}
Note that the operator $S$ also depends on the choice of $p$ and on the data $f$ and $g$. 
For ease of presentation, we will emphasize this dependence only if necessary. 

\subsection{Continuity}\label{sec:weakClosedness}

Let us start with presenting some results about the continuity of $S$ with respect to the strong and weak topologies. 
The latter case will play a fundamental role in the analysis of the optimal control problem later on.

\begin{theorem}[Continuity]\label{thm:continuity}
 Let $1<p,q<\infty$ and assume that $f \in L^q(\D)$ and $g\in L^q(\Gamma_-;|s\cdot n|)$.
Then $S$ is continuous as mapping from $L^p(\R) \times L^p(\R)$ to $\V^q$.
\end{theorem}
\begin{proof}
Let $(\mu,\sigma)\in D(S)$ and $\{(\mu^n,\sigma^n)\} \subset D(S)$ such that $(\mu^n,\sigma^n)\to (\mu,\sigma)\in L^p(\R)\times L^p(\R)$. Furthermore, denote by $\phi$ and $\phi^n$ the solutions of \eqref{eq:op_eq}--\eqref{eq:op_bc} with parameters $(\mu^n,\sigma^n)$ and $(\mu,\sigma)$, respectively. Then
\begin{align*}
  \big(\A + \C(\mu^n,\sigma^n)\big) (\phi^n-\phi) = (\mu-\mu^n)\phi + (\sigma-\sigma^n)\Theta\phi.
\end{align*}
Since $\mu^n\to \mu$ in $L^p(\R)$, we can choose a subsequence, again denoted by $\{\mu^n\}$, such that $\mu^n\to \mu$ a.e.\@ in $\R$ and consequently $\mu_n\phi \to \mu \phi$ a.e.\@ in $\D$. Since $|\mu_n \phi|\leq C |\phi|$ is uniformly bounded, Lebesgue's dominated convergence theorem ensures $(\mu-\mu^n)\phi \to 0$ in $L^q(\D)$. Similarly, $(\sigma-\sigma^n)\Theta\phi\to 0$ in $L^q(\R\times\S)$. The uniform a-priori estimate of Theorem~\ref{thm:existence} then yields $\phi^n\to\phi$ in $\V^q$.
\end{proof}

We will show next, that the parameter-to-solution map is also continuous in the weak topology. 
This directly implies the weak-lower semi-continuity of the Tikhonov functional and thus 
yields the well-posedness of the regularization method.
The proof of the result heavily relies on the compactness provided by the averaging lemma.

\begin{theorem}[Weak continuity]\label{thm:weak_continuity}
Let $1<p,q<\infty$ and assume that $f \in L^q(\D)$ and $g\in L^q(\Gamma_-;|s\cdot n|)$. 
Then $S$ is weakly continuous, i.e., if $ D(S) \ni (\sigma^n,\mu^n)\rightharpoonup (\sigma,\mu)$ in $L^p(\R)\times L^p(\R)$, then 
$(\sigma,\mu)\in D(F)$ and 
$S(\sigma^n,\mu^n) \rightharpoonup S(\sigma,\mu)$ in $\V^{q}$.
\end{theorem}
\begin{proof}
  Since $D(S)$ is closed and convex, $D(S)$ is weakly closed and $(\mu,\sigma)\in D(S)$. 
  Now let $\phi_n,\phi \in \V^p$ denote the unique solutions of \eqref{eq:op_eq}--\eqref{eq:op_bc} with parameters $\mu^n$, $\sigma^n$ and $\mu$, $\sigma$, respectively. 
  Then the difference $\phi^n-\phi$ satisfies the transport problem 
  \begin{align} \label{eq:difference}
    (\A+\C(\mu,\sigma))(\phi-\phi^n) = \tilde f^n \quad \text{in } \R\times\S, \qquad \phi-\phi^n=0 \quad \text{on } \dD_-
  \end{align}
  with right-hand side defined by
  \begin{align*}
  \tilde f^n = (\mu^n-\mu) \phi^n + (\sigma^n - \sigma) (\phi^n - \Theta \phi^n).
  \end{align*}
  By Theorem~\ref{thm:existence}, the operator $\A+\C$ is continuously invertible.  
  It thus remains to prove that $\tilde f^n \rightharpoonup 0$.
  Multiplying the first term with $\psi\in C^\infty_0(\D)$ and integrating yields
  \begin{align*}
    \int_{\R\times\S} (\mu^m-\mu)\phi^n \psi\d(r,s) 
      &= \int_\R (\mu-\mu^n) \int_\S \phi^n(r,s)\psi(r,s)\ds\dr =: I^n(\psi).
  \end{align*}
  Now by Lemma~\ref{lem:scatteringcompactness}, we obtain $\int_\S \phi^n \psi \ds = \Theta(\phi^n \psi) \to \Theta(\phi \psi)$ strongly in $L^p(\D)$. From this we conclude that $I^n(\psi) \to 0$ and as a consequence $(\mu^n-\mu)\phi^n \rightharpoonup 0$.  The term involving $\sigma^n-\sigma$ can be treated in a similar way.
\end{proof}

For the following quantitative estimate, we require some slightly stronger assumptions on the  source terms. This kind of regularity seems to be necessary since due to its hyperbolic type the transport equation does not possess a regularizing effect.

\begin{theorem}\label{thm:lipschitz}
 Let $f \in L^\infty(\D)$ and $g\in L^\infty(\dD_-)$. 
Then for any $1\leq q\leq p\leq \infty$ the operator
$S$ is Lipschitz continuous as a mapping from $L^p(\R) \times L^p(\R)$ to $\V^q$.
\end{theorem}
\begin{proof}
 Let $(\mu,\sigma),(\tilde\mu,\tilde\sigma)\in D(S)$ and denote by $\phi,\tilde\phi\in\V^q$ the corresponding solutions of the transport problem \eqref{eq:op_eq}--\eqref{eq:op_bc}. 
 The difference $\tilde\phi-\phi$ then satisfies \eqref{eq:difference} with right-hand side $\tilde f=(\tilde \mu - \mu) \tilde \phi + (\tilde \sigma - \sigma) (\tilde \phi - \Theta \tilde \phi)$.
 Using Theorem~\ref{thm:existence} we obtain 
\begin{align*}
 \|\phi-\tilde\phi\|_{\V^q}
\leq C \big( \|\tilde\mu-\mu\|_{L^q(\R)} + \|\tilde\sigma-\sigma\|_{L^q(\R)} \big) \|\tilde \phi\|_{L^\infty(\D)}.
\end{align*}
Due to the regularity of the data $f$ and $g$, we have $\tilde \phi \in \V^\infty$, which completes the prove.
\end{proof}

\subsection{Differentiability}\label{sec:differentiability}

As a next step, we investigate differentiability of the parameter-to-solution map.
We call a parameter pair $(\hat\mu,\hat\sigma)\in L^p(\R)\times L^p(\R)$ an admissible variation for 
$(\mu,\sigma)\in D(S)$, if the perturbed parameters
$(\mu,\sigma)+t(\hat\mu,\hat\sigma)\in D(F)$ for  $|t|\ll 1$. 
\begin{theorem}\label{thm:gateaux}
Let $1\leq q\leq p\leq \infty$ and let $f \in L^\infty(\D)$ and $g\in L^\infty(\dD_-)$. For $(\mu,\sigma)\in D(S)$ and admissible variation $(\hat\mu,\hat\sigma) \in L^p(\R) \times L^p(\R)$, let 
$S'(\mu,\sigma)[\hat\mu,\hat\sigma] = w \in \V^q$ be defined as the unique solution of
\begin{align}\label{eq:sensitivity}
 \A w + \C w = \tilde f \quad \text{in } \D, 
\qquad  
   w =0 \quad \text{on } \dD_-
\end{align}
with $\tilde f = - \C(\hat\mu,\hat\sigma)\phi$ and $\phi\in\V^q$ solving \eqref{eq:op_eq}--\eqref{eq:op_bc} 
with parameters $(\mu,\sigma)$. Then, there holds 
\begin{align}\label{eq:est_der}
\|S'(\mu,\sigma)[\hat\mu,\hat\sigma]\|_{\V^q} \leq C \big(\|\hat\mu\|_{L^p(\D)} + \|\hat\sigma\|_{L^p(\D)} \big) \|g\|_{L^\infty(\dD_-)}.
\end{align}
\end{theorem}
\begin{proof}
Let $\phi_t\in\V^q$ denote the solution of \eqref{eq:op_eq}--\eqref{eq:op_bc} for parameters $(\mu,\sigma)+t(\hat\mu,\hat\sigma)\in D(S)$ and $t$ sufficiently small and let $w_t := (\phi_t-\phi)/t$.
Then
\begin{alignat*}{3}
 \A (w_t-w) + \C (w_t-w) &= \C(\hat\mu,\hat\sigma)(\phi-\phi_t) &\quad&\text{in } \D,
\end{alignat*}
and $w_t-w = 0$ on $\dD_-$. By Theorem~\ref{thm:existence} we thus obtain
\begin{align*}
 \|w_t-w\|_{\V^q} \leq C \big(\|\hat\mu\|_{L^p(\D)} + \|\hat\sigma\|_{L^p(\D)} \big) \|\phi-\phi_t\|_{L^\infty(\D)}.
\end{align*}
The continuity of the parameter-to-solution map, and the integrability condition on the data, yields 
$\phi_t \to \phi$ in $L^\infty(\D)$ from which we conclude that $w_t\to w$ in $\V^q$ as $t\to 0$.
The estimate \eqref{eq:est_der} follows again from Theorem~\ref{thm:existence}. 
\end{proof}
One can see from \eqref{eq:sensitivity} that $S'$ depends linearly on the variation $(\hat\mu,\hat\sigma)$. 
By the continuous extension principle, the operator $S'(\mu,\sigma)$ can then be extended to a bounded linear 
operator $S'(\mu,\sigma):L^p(\R)\times L^p(\R)\to \V^q$, which we call the derivative of $S$ in the following.
\begin{theorem}\label{thm:der_lip}
 Let $2\leq p\leq \infty$ and $1 \le q \le p/2$ and assume that $f \in L^\infty(\D)$ and $g\in L^\infty(\Gamma_-)$. 
Then $S'$ is Lipschitz continuous, i.e., for $(\mu_1,\sigma_1)$,  $(\mu_2,\sigma_2)\in D(S)$ there holds
 \begin{align*}
 & \| S'(\mu_{1},\sigma_{1}) - S'(\mu_{2},\sigma_{2})\|_{\mathcal{L}(L^{p}(\R)\times L^{p}(\R);\V^{q})} \\
 & \qquad \qquad \qquad \leq L \big(\|\mu_{1}-\mu_{2}\|_{L^{p}(\R)}+ \|\sigma_{1}-\sigma_{2}\|_{L^{p}(\R)} \big) \|g\|_{L^\infty(\dD_-)}.
 \end{align*}
\end{theorem}
\begin{proof}
 Let $(\mu_{i},\sigma_{i})\in D(S)$, $i=1,2$, and let $w_{i}\in\V^q$, $i=1,2$, be the solutions of the sensitivity problems in Theorem~\ref{thm:gateaux} for some admissible direction $(\hat\mu,\hat\sigma)\in L^p(\R)\times L^p(\R)$.
 Then $w_1-w_2$ satisfies \eqref{eq:sensitivity} with $\tilde f=-\C(\hat\mu,\hat\sigma)(\phi_1-\phi_2)-\C(\mu_1-\mu_2,\sigma_1-\sigma_2)w_2$.
 Using H\"older's inequality the two parts of $\tilde f$ can be estimated individually by
 \begin{align}
 \| \C(\hat\mu,\hat\sigma)(\phi_1-\phi_2) \|_{L^{q}(\D)} &\leq C \big(\|\hat\mu\|_{L^{p}(\R)}+\|\hat\sigma\|_{L^{p}(\R)}\big) \|\phi_1-\phi_2\|_{L^{p}(\D)},\label{eq:sens1}\\
 \| \C(\mu_1-\mu_2,\sigma_1-\sigma_2)w_2\|_{L^{q}(\D)} &\leq C \big(\|\mu_1-\mu_2\|_{L^{p}(\R)}+\|\sigma_1-\sigma_2\|_{L^{p}(\R)}\big)\|w_2\|_{L^{p}(\D)}.\label{eq:sens2}
\end{align}
Using Theorem~\ref{thm:lipschitz} and Theorem~\ref{thm:gateaux} we then obtain via the triangle inequality 
\begin{align*}
  \|\tilde f\|_{L^q(\R\times\S)}\leq C \big(\|\hat\mu\|_{L^{p}(\R)}+\|\hat\sigma\|_{L^{p}(\R)}\big) \big(\|\mu_1-\mu_2\|_{L^{p}(\R)}+\|\sigma_1-\sigma_2\|_{L^{p}(\R)}\big).
\end{align*}
The Lipschitz estimate now follows from the a-priori estimates stated in Theorem~\ref{thm:existence}.
\end{proof}
%
Differentiability of $S$ has already been proven in \cite{DiDoNaPaSi02},
but under more restrictive assumptions and only for $p=\infty$, which turns out to be the simplest case. 
The proofs of \cite{DiDoNaPaSi02} cannot be applied to the more general setting considered here.
By carefully inspecting the estimates \eqref{eq:sens1}--\eqref{eq:sens2}, using 
assumptions (A2)--(A3), H\"older's inequality, and interpolation, we obtain 
\begin{corollary}\label{cor:der_hoelder}
 Let $1\leq q <\infty$ and $q<p\leq 2q$ and assume that $f \in L^\infty(\D)$ and $g\in L^\infty(\Gamma_-)$. Then $S'$ is H\"older continuous with H\"older exponent $\frac{p-q}{q}$.
\end{corollary}
This estimate will allow us to obtain convergence of iterative minimization algorithms under very general conditions.
With the same techniques as used to prove Theorem~\ref{thm:der_lip}, 
one can also analyze higher order derivatives. 
For later reference let us state a result about the existence of the Hessian.
\begin{theorem}\label{thm:hessian}
 Let $p=3q$ for some $1\leq q\leq \infty$ and assume that $f \in L^\infty(\D)$ and $g\in L^\infty(\Gamma_-)$. 
 Then $S : D(S) \subset L^p(\R) \times L^p(\R) \to \V^q$ is twice continuously differentiable and $S''$ is given by
 \begin{align*}
   S''(\mu,\sigma)[(\hat\mu_{1},\hat\sigma_{1}), (\hat\mu_{2},\hat\sigma_{2})] = H,
 \end{align*}
 where $H\in \V^q$ is the unique solution of
 \begin{alignat*}{3}
    \A H + \C H &= \C(\hat\mu_{1},\hat\sigma_{1})w(\hat\mu_{2},\hat\sigma_{2}) +\C(\hat\mu_{2},\hat\sigma_{2}) w(\hat\mu_{1},\hat\sigma_{1}) &\quad&\text{in } \D, &\\
     H &=0 &\quad&\text{on } \dD_-.
 \end{alignat*} 
  Moreover, $S''(\mu,\sigma)$ is Lipschitz continuous w.r.t. its arguments and
 \begin{align*}
  &\| S''(\mu_{1},\sigma_{1}) - S''(\mu_{2},\sigma_{2})\|_{\mathcal{L}(L^{p}(\R)\times L^{p}(\R), L^{p}(\R)\times L^{p}(\R); \V^{q})}\\ 
  &\qquad \qquad \leq C \big(\|\mu_{1}-\mu_{2}\|_{L^{p}(\R)}+ \|\sigma_{1}-\sigma_{2}\|_{L^{p}(\R)} \big),
 \end{align*}
  with  $C$ depending only on the domain, the bounds for the parameters, and the data.
\end{theorem}
Like above, the Hessian should first be defined for admissible parameter variations and then be extended to a bounded bilinear map. The estimate then follows in the same way as the Lipschitz estimate for the first derivative. We will utilize the properties of the Hessian to show local convexity of the regularized functional \eqref{eq:tikhonov} in a Hilbert space setting.
\section{The optimal control problem} \label{sec:OCproblem}

Let us recall the definition of the optimal control problem 
\begin{align*}
\| BS(\mu,\sigma)- \md\|_{L^q(\dR)}^q  + \alpha \|\mu-\mu_0\|^p_{L^p(\R)} + \alpha \|\sigma-\sigma_0\|^p_{L^p(\R)} 
  \to \min_{(\mu,\sigma) \in D(S)},
\end{align*}
defined by minimizing the Tikhonov functional for some $\alpha \ge 0$.
Based on the results about the mapping properties of the parameter to solution map $S$ and the observation operator $B$, 
we will now comment on the existence and stability of minimizers. The arguments are rather standard, and we only sketch the main points. Let us refer to \cite{EHN96,EnKuNeu89} for details and proofs.


\subsection{Existence of Minimizers}
By weak continuity of $S$ and weak lower semi-continuity of norms, the Tikhonov functional is weakly lower semi-continuous and bounded from below. Due to the box constraints and the reflexivity of $L^p$, $1<p<\infty$, the domain $D(S)$ is weakly compact. This yields the existence of a minimizer $(\mu_\alpha,\sigma_\alpha)$ for any $\alpha\geq 0$.
  
  
\subsection{Stability of Minimizers}
The minimizers  are stable w.r.t. perturbations in the following sense: For $\alpha_n\to\alpha \geq 0$ and $\md^{n}\to \md$ there exists a sequence of minimizers $(\mu_{\alpha_n},\sigma_{\alpha_n})$ converging weakly to a minimizer $(\mu_\alpha,\sigma_\alpha)$. This follows from the weak compactness of $D(S)$ and weak continuity of $S$. 
If $\alpha>0$, then we can obtain strong convergence. 


\subsection{Convergence of Minimizers}
From the stability result, we already deduce that subsequences of minimizers $(\mu_{\alpha_n},\sigma_{\alpha_n})$ converge weakly towards a minimizer of the $L^p$-norm residual of equation \eqref{eq:ip_informal} if $\alpha_n \to 0$.
If the inverse problem is solvable and if $\alpha_n\to 0$ and $\|\md^n-\md\|_{L^p}^p/\alpha_n \to 0$, 
then convergence is strong and the limit is a solution of \eqref{eq:ip_informal}.


\subsection{Remarks and generalizations}
Note that, in general, uniqueness of solutions for the inverse problem \eqref{eq:ip_informal} 
or of minimizers for the optimal control problem \eqref{eq:tikhonov}  
cannot be expected. We will discuss this issue in more detail in the next section. 
Also note that, with the same arguments as above, we can analyze minimization problems of the form
  \begin{align*}
      \| BS(\mu,\sigma)- \md\|_{L^q(\dR)}^q  + \alpha R(\mu,\sigma)
  \to \min,
  \end{align*}
  where $R$ is some more general regularization functional. 
  One particular choice $R(\mu,\sigma) = \|\mu - \mu_0\|^2_{H^1(\D)} + \|\sigma - \sigma_0\|^2_{H^1(\D)}$
  will be considered in more detail in the next section. 
  Total variation regularization $R(\mu,\sigma) = |\mu|_{TV} + |\sigma|_{TV}$ is frequently used in image reconstruction; for an analysis see for instance \cite{AcarVogel}.
  Due to the continuous embedding of $H^1$ and $BV$ in certain $L^p$ spaces, the statements about existence, stability, and convergence of minimizers made above also hold true for these choices. Our results thus generalize those of \cite{TangHanHan2013}.
  Note however, that in dimension $d=3$, we cannot obtain Lipschitz- or H\"older continuity of the derivative $S'$ for $TV$-regularization, while for $H^1$ we even obtain Lipschitz continuous second derivatives. This is our guideline for the setting of the next section.

\section{Iterative minimization algorithms}\label{sec:iterative}

To ensure convergence of minimization algorithms, one has to impose some more restrictive conditions. 
In order to motivate the crucial assumptions, let us recall a basic convergence rate result from 
nonlinear regularization theory \cite{EHN96,EnKuNeu89}.
To simplify the presentation, we restrict ourselves to a Hilbert space setting and consider the Tikhonov functional
\begin{align} \label{eq:tikhonov2}
 \| BS(\mu,\sigma)- \md\|_{L^2(\dR)}^2  + \alpha \| \mu-\mu_0\|_{H^1(\R)}^2 + \alpha \| \sigma-\sigma_0\|_{H^1(\R)}^2.
\end{align}
Note that due to the continuous embedding of $H^1$ into $L^6$ in dimension $d \le 3$, we can use all properties of $S$ 
derived in Section~\ref{sec:properties} for $q=2$ and $p \le 6$. In particular, we infer from Theorem~\ref{thm:der_lip} and Theorem~\ref{thm:hessian} that $S$ has Lipschitz-continuous first and second derivatives.

\subsection{Convergence Rates for Minimizers}
It is well-known that quantitative estimates for convergence can only be obtained under some kind of source condition. 
We therefore assume in the following that there exists some $w\in \V^2$ such that 
\begin{align}\label{eq:source}
    (\mu^\dagger,\sigma^\dagger)-(\mu_0,\sigma_0) = S'(\mu^\dagger,\sigma^\dagger)^* w,\qquad L\|w\|_{\V^2}<1,
\end{align}
where $(\mu^\dagger,\sigma^\dagger)$ solves \eqref{eq:ip_informal} and $L$ is the Lipschitz constant of $S'$; see Theorem~\ref{thm:der_lip}.
From the abstract theory of nonlinear Tikhonov regularization \cite{EHN96,EnKuNeu89}, we deduce that
  \begin{align*}
    \|(\mu_\alpha,\sigma_\alpha)-(\mu^\dagger,\sigma^\dagger)\|_{H^1(\R)\times H^1(\R)} = \O(\sqrt{\alpha})\quad\text{and}\quad \|BS(\mu_\alpha,\sigma_\alpha)-\md\|_{L^2(\dR)} = \O(\alpha),
  \end{align*}
where $(\mu_\alpha,\sigma_\alpha)$ are corresponding minimizers of the Tikhonov functional with $\alpha>0$.  
Note that the best possible rate one could expect for the error in the parameters is $o(1)$, and for the residual is $\mathcal{O}(\sqrt{\alpha})$, if \eqref{eq:source} is not fulfilled.

\subsection{An iterative algorithm for computing a minimizer}\label{sec:PGN}
For minimizing the Tikhonov functional \eqref{eq:tikhonov2}, we consider a projected Gau\ss-Newton (PGN) method. To ease the notation, we use $x=(\mu,\sigma)$ and $F(x)=B S(\mu,\sigma)$. The method then reads 
\begin{align*}
\hat x_{n+1} &= x_n + \big( F'(x_n)^* F'(x_n) + \alpha_k I)^{-1} \big[F'(x_n)^* (\md - F(x_n)) + \alpha_k (x_0 - x_n) \big] \\
x_{n+1} &= P_{D(S)}(\hat x_{n+1}). 
 \end{align*}
Here $P_{D(S)}$ denotes the metric projection onto $D(S)$ with respect to the $H^1$-norm
and $F'(x)^* = S'(\mu,\sigma)^* B^*$ is the Hilbert space adjoint of the linearized parameter-to-measurement operator. As usual, $F'(x)^* w$ can be computed via the solution of an adjoint problem similar to \eqref{eq:sensitivity}.
A detailed analysis of the PGN iteration in the framework of iterative regularization methods can be found \cite{BakKok04,KalNeu06}.
Here, we consider this algorithm for the approximation of minimizers $x_\alpha = (\mu_\alpha,\sigma_\alpha)$ of the Tikhonov functional \eqref{eq:tikhonov2}. 
To promote global convergence, we choose a geometrically decaying sequence $\alpha_n=\max\{\frac{\alpha_0}{2^n},\alpha\}$ of regularization parameters. 
If the source condition \eqref{eq:source} holds with $\|w\|$ sufficiently small, 
then
 \begin{align*}
    \|x_n-x^\dag\|_{H^1(\R)\times H^1(\R)} \le C \sqrt{\alpha_n} \|w\|_{\V^2}
\quad \text{and} \quad 
   \|F(x_n) - \md\|_{L^2(\dR)} \le C \alpha_n \|w\|_{\V^2}
 \end{align*}
with a constant $C$ not depending on $\alpha$ or $w$.
For $\alpha = 0$, we recover the usual convergence rate statement of the iterative regularization method without data noise \cite[Chapter~4]{BakKok04}. For $\alpha>0$, the iteration is bounded but convergence is not so clear.

\subsection{Local convexity and convergence to minimizers}
We will now explain that for $\alpha>0$ and under the source condition \eqref{eq:source}, the PGN iteration converges to a local minimizer $x_\alpha$ of the Tikhonov functional.
Consider the Hessian of the Tikhonov functional given by 
\begin{align*}
H(x) = F''(x)^* (F(x) - \md) + F'(x)^* F'(x) + \alpha I.   
\end{align*}
One can easily see that, if $F$ is two-times differentiable and the norm of the residual $F(x)-\md$ is sufficiently small, such that $\|F''(x)^* (F(x)-\md)\| < \alpha$, then the Hessian is positive definite. 
Now, by the Lipschitz estimate for the first derivative we deduce that $\|F''(x)\| \le L \|B\|$, 
and from the convergence rate estimates for nonlinear Tikhonov regularization we have
 $\|F(x_\alpha)-\md\| \le C \alpha \|w\|$. 
Hence we conclude that, if the source condition \eqref{eq:source} is valid and $\|w\|$ is sufficiently small, then the Tikhonov functional is locally convex in a neighborhood of the minimizers $x_\alpha$.
From the estimates for $\|x_\alpha - x^\dag\|$ and $\|x_n-x^\dag\|$ and by the Lipschitz estimate for the first derivative, one can actually conclude that the region of convexity is always reached after a finite number of iterations. For a detailed analysis using similar arguments see \cite{Ram03}. 
In the area of convexity, the linear convergence follows with standard arguments.

\subsection{Remarks and Extensions}
Using the abstract theory of regularization methods in Banach spaces \cite{HoKaPoSch07}, 
the statements of the section can in principle be extended to the $L^p$-$L^q$ setting considered earlier;
see also \cite{Resmerita2005,ScherzerVariationalMethodsInImaging2009,SchusterKaltenbacherHofmannKazimierski2012}. The required convergence rates results for the GN method in Banach spaces have been established in \cite{KaltenbacherHofmann2010,KalSchSch09}.
At the end of our discussion, let us mention that also projected gradient methods in combination with appropriate rules for the choice for the stepsize can be used for minimizing the Tikhonov functional. For these methods, convergence to stationary points can be established even without a source condition and merely under H\"older continuity of the derivative \cite{HiPiUlUl09}. The same holds true for the PGN method \cite{EggSch11}.

\section{Computational Experiments} \label{sec:numerics}

To illustrate the theoretical results of the previous sections, we will present some numerical experiments in the following.

\subsection{Discretization}
For discretization of the radiative transfer problem \eqref{eq:rte}--\eqref{eq:rte_bc} we employ the
$P_N$-FEM method. This is a Galerkin approximation using a truncated spherical harmonics expansion with respect to the direction $s$ and a mixed finite element approximation for the corresponding spatially dependent Fourier coefficients. 
Due to the variational character of the method, one can systematically obtain consistent discretizations of the operator parameter-to-solution operator $S$, its derivative $S'$, and the adjoint $(S')^*$. Let us refer to \cite{EggSch10:3,LewisMiller84,WrightArridgeSchweiger07} for an analysis of the method and details on the implementation.

\subsection{Test example and choice of parameters}
We consider the setup depicted in Figure~\ref{fig:setup}: The computational domain $\R$ is a two-dimensional circle with radius $25$ mm. The absorption parameter $\mu$ is in the range of $0.005$ mm$^{-1}$ to $0.04$ mm$^{-1}$. The scattering $\sigma$ ranges from $5$ mm$^{-1}$ to $30$ mm$^{-1}$. This order of magnitude is typical for applications in optical tomography \cite{Arridge99}. The data $\md\in \RR^{16\times 16}$ are generated by sequentially illuminating the object by one of the sources $g_j$ and recording the outgoing light on the $i$th detector for prescribed parameters $\mu$ and $\sigma$, i.e.
\begin{align}\label{eq:measurement}
  \md_{ij} = \int_{\Sigma_i} B \phi_j(r) \d r.
\end{align}
Here $\phi_j$ is the photon density generated by the $j$th source and $\Sigma_i\subset\dR$ models the area of the $i$th detector; see Figure~\ref{fig:setup} for the arrangements of sources and detectors.
For our numerical experiments, we choose a sequence of regularization parameters $\alpha_n = \max\{\frac{\alpha_0}{2^n},\alpha_{\min}\}$ with $\alpha_0=\frac{1}{100}$ and $\alpha_{\min}=10^{-10}$. 
As initial guess, we use the constant functions $\mu_0=0.015$ mm$^{-1}$ and $\sigma_0=15$ mm$^{-1}$.

\begin{figure}[ht]
  \includegraphics[width=0.33\textwidth]{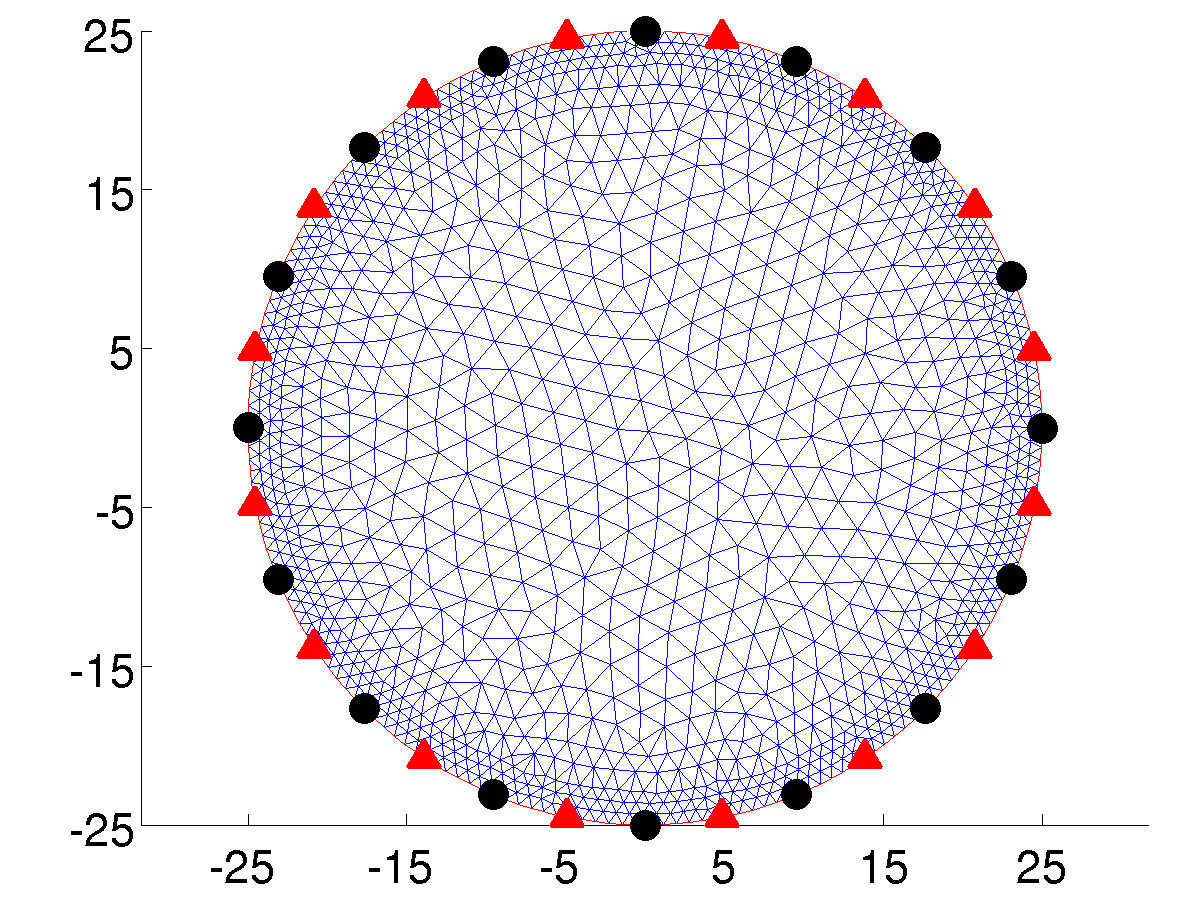}\hfill
  \includegraphics[width=0.33\textwidth]{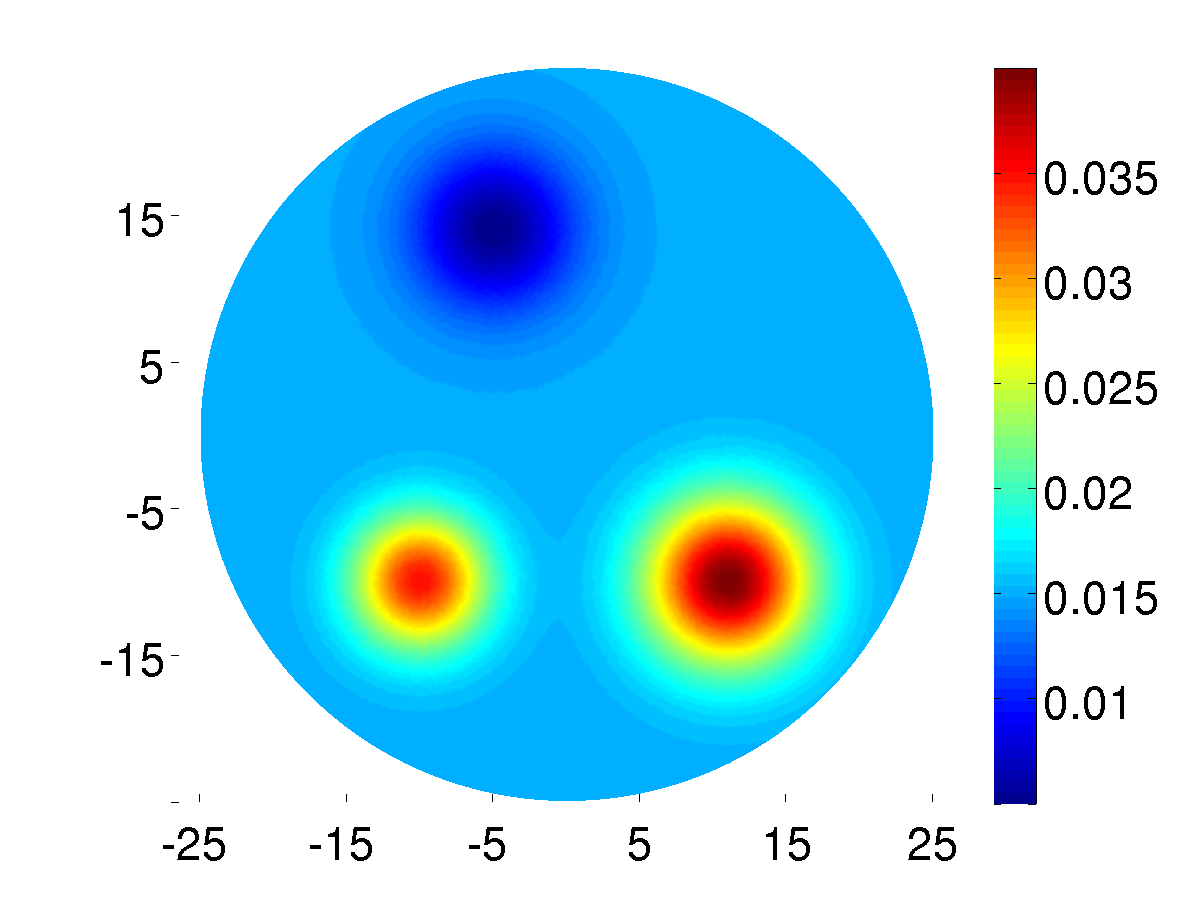}\hfill
  \includegraphics[width=0.33\textwidth]{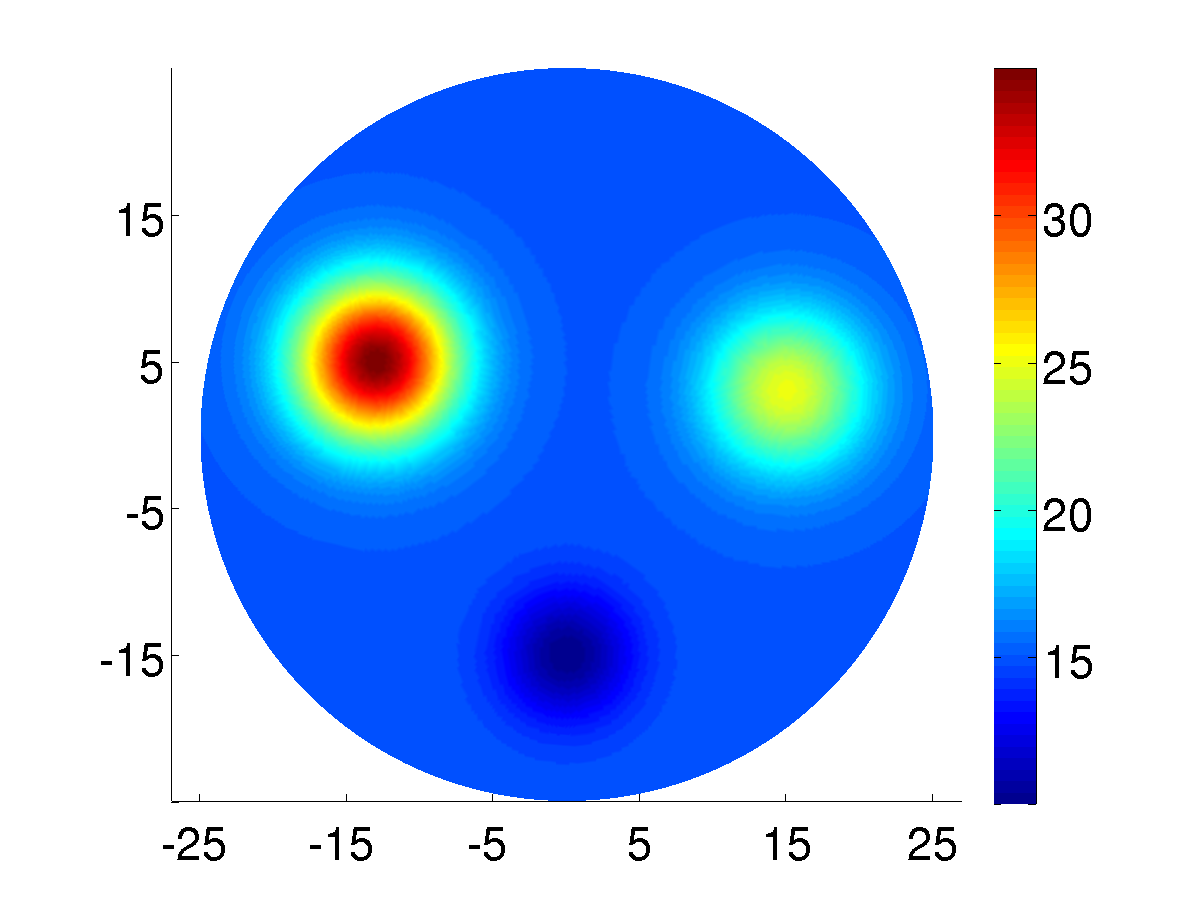}
  \caption{\label{fig:setup} Left: Grid with $1287$ vertices, blue circles denote the $16$ source positions, red triangles denote $16$ detector positions. Middle: True distribution of $\mu$. Right: True distribution of $\sigma$.}
\end{figure}


\subsection{Generation of Data and Non-uniqueness}
Note that our choice of parameters $\mu$ and $\sigma$ depicted in Figure~\ref{fig:setup} 
cannot be expected to satisfy the source condition \eqref{eq:source}.
To be able to observe convergence rates, we therefore compute in a first step a minimizer $(\mu^\dagger,\sigma^\dagger):=(\mu_{\alpha_{\min}},\sigma_{\alpha_{\min}})$ of the Tikhonov functional with $\alpha_{\min}=10^{-10}$. 
The result of this preprocessing step is depicted in Figure~\ref{fig:xad}. 
\begin{figure}[ht]
  \includegraphics[width=0.4\textwidth]{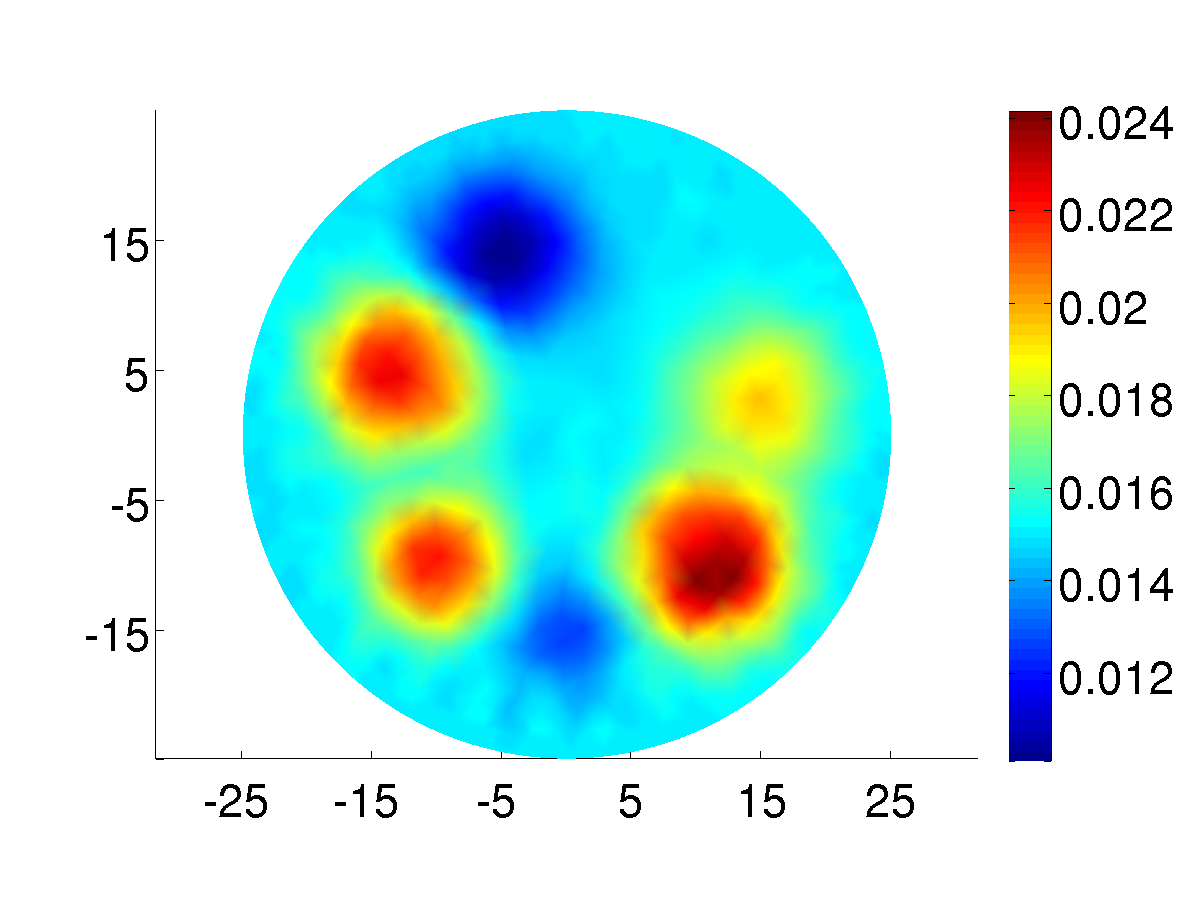}\quad
  \includegraphics[width=0.4\textwidth]{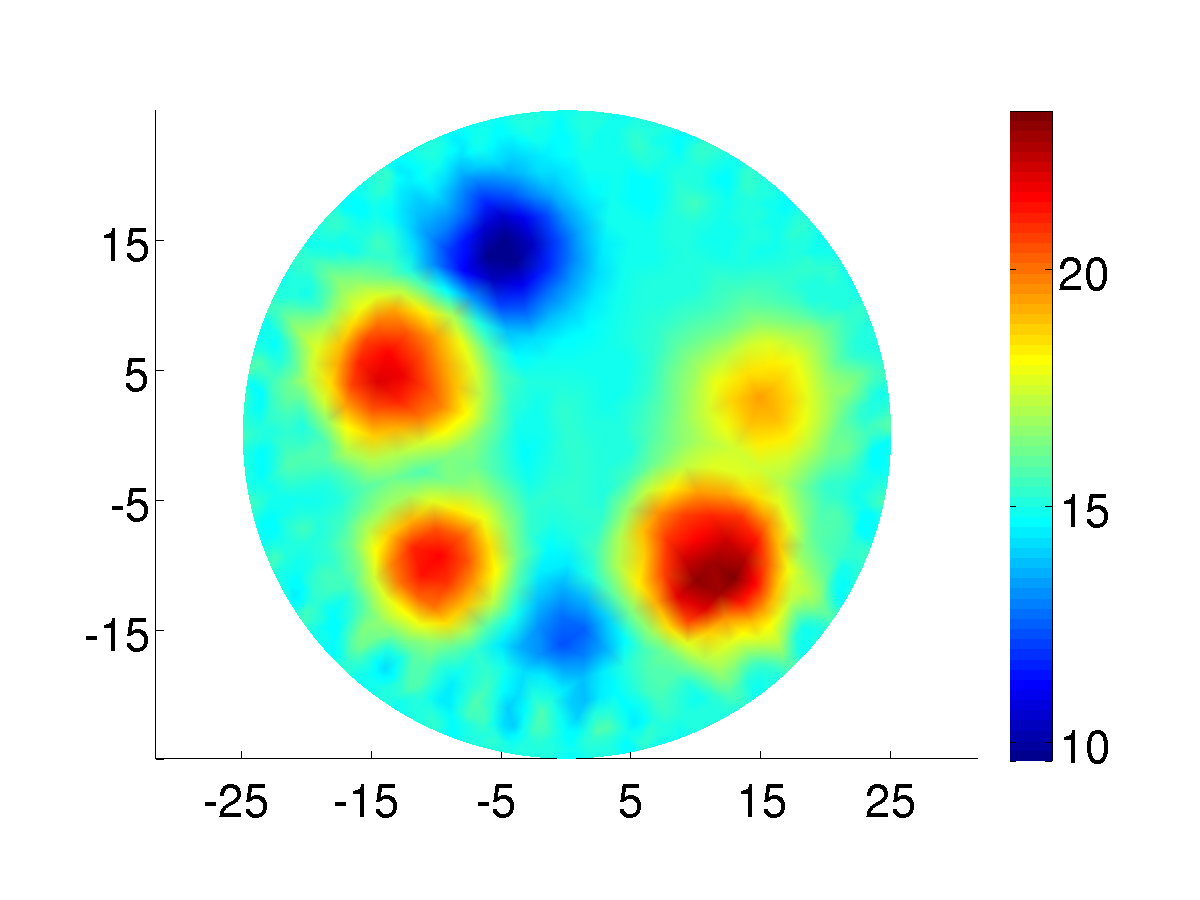}
  \caption{\label{fig:xad} Calibrated parameters $\mu^\dagger$ (left) and $\sigma^\dagger$ (right) obtained by minimizing the Tikhonov functional for initial guess $\mu_0=0.015$ and $\sigma_0=15$, $\alpha=10^{-10}$ and data $BS(\mu,\sigma)$ from Figure~\ref{fig:setup}.}
\end{figure}
Let us mention that we obtain different reconstructions $(\mu^\dagger,\sigma^\dagger)$ 
when changing the initial value $(\mu_0,\sigma_0)$, which is a clear indication of non-uniqueness in for the inverse problem \eqref{eq:ip_informal}; see also \cite{Arridge99} for a theoretical explanation.
Using the calibrated parameter $(\mu^\dag,\sigma^\dag)$ as truth-approximation, we then compute the measurements 
$\md = B S(\mu^\dag,\sigma^\dag)$ as in \eqref{eq:measurement}. The relative error in the data corresponding to the 
parameters depicted in Figure~\ref{fig:setup} and \ref{fig:xad} is less then $0.05$\%. 
This indicates the ill-posedness and possible non-uniqueness for the inverse problem.

\subsection{Convergence rates for minimizers}
In a first numerical test, we want to demonstrate the convergence of the minimizers 
$(\mu_\alpha,\sigma_\alpha)$ of the Tikhonov functional \eqref{eq:tikhonov2} towards the correct parameter pair $(\mu^\dag,\sigma^\dag)$ generated in the preprocessing step. 
We denote by
$$
{\rm res}_\alpha = \|BS(\mu_\alpha,\sigma_\alpha)-\md\|_{2},\qquad
{\rm err}_\alpha = \|(\mu_\alpha,\sigma_\alpha)-(\mu^\dagger,\sigma^\dagger)\|_{H^1(\R)},
$$
the observed residuals and errors in the regularized solutions. 
The convergence rates for the residual and the error can be seen in Figure~\ref{fig:rate}. 

\begin{figure}[ht]
  \includegraphics[width=0.4\textwidth]{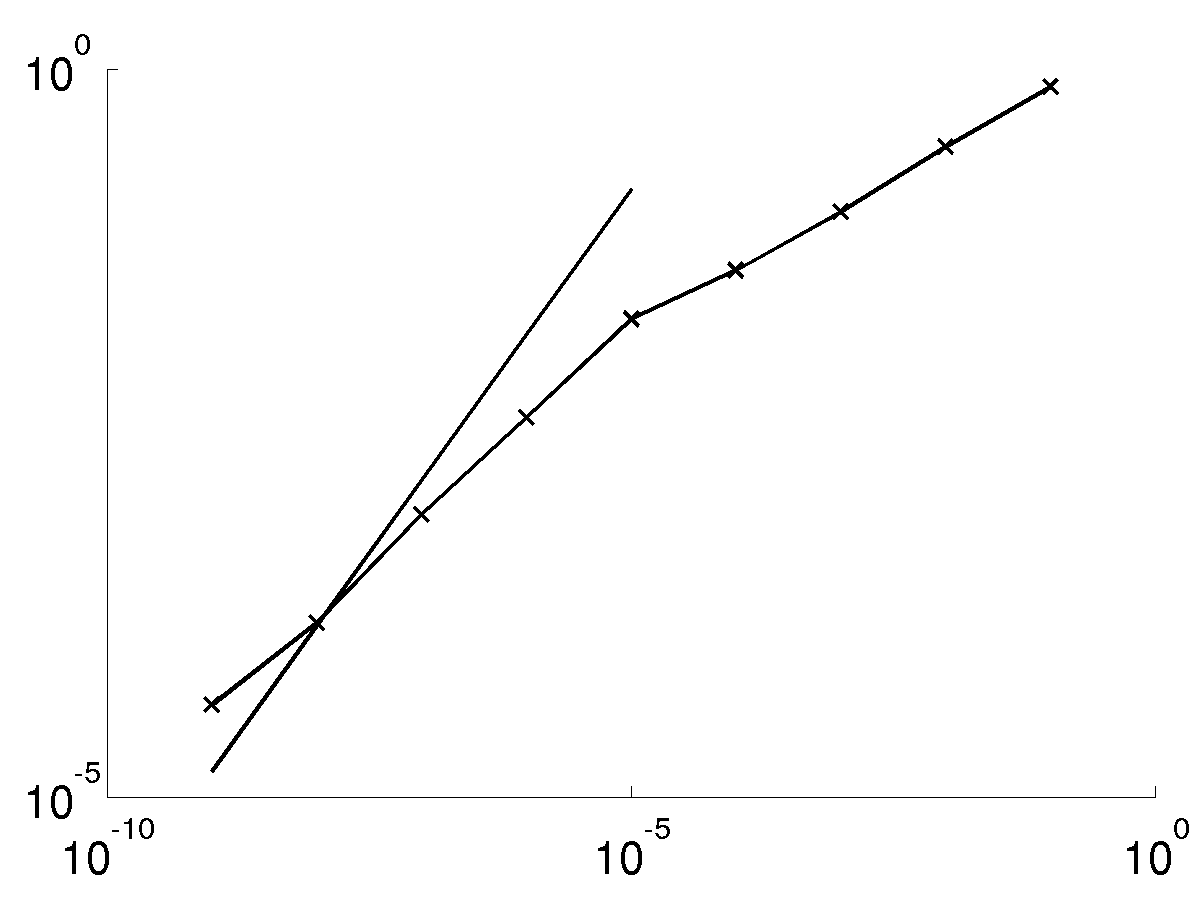}\qquad
  \includegraphics[width=0.4\textwidth]{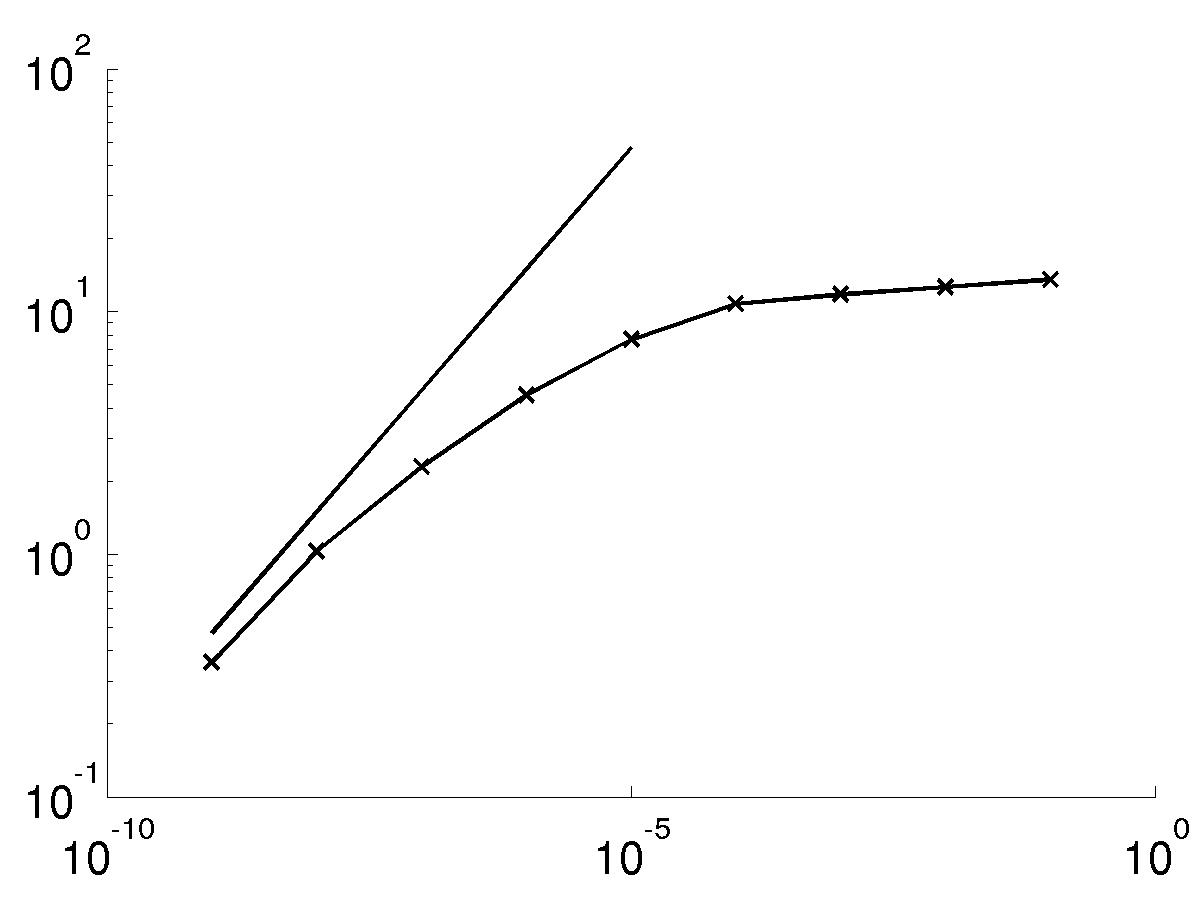}
  \caption{\label{fig:rate} Rates of convergence for minimizers of the Tikhonov functional. Left: ${\rm res}_\alpha$ (crosses) and $\mathcal{O}(\alpha)$ for $\alpha=10^{-n}$ and $n\in\{1,\ldots,9\}$. Right: ${\rm err}_\alpha$ (crosses) and $\mathcal{O}(\sqrt{\alpha})$.}
\end{figure}

As predicted by theory, we observe the asymptotic rate $\mathcal{O}(\sqrt{\alpha})$ for the error ${\rm err}_\alpha$. 
The convergence rate for the residuals ${\rm res}_\alpha$ is slightly less than the expected rate $\mathcal{O}(\alpha)$.

\subsection{Convergence of PGN method for $\alpha$ fixed}
With the second experiment, we would like to demonstrate the linear convergence of the PGN method to 
the minimizer of the Tikhonov functional.
To do so, we compute for $\alpha=10^{-5}$ the minimizers $(\mu_\alpha,\sigma_\alpha)$ by iterating the PGN method until convergence.
We then restart the iteration to create a sequence $(\mu_n,\sigma_n)$ of PGN iterates defined as in Section~\ref{sec:PGN}
with $\alpha_n=\max(\frac{1}{100}\frac{1}{2^n},\alpha)$. 
The residuals and the errors in the $n$th iteration given by
$$
{\rm res}_n^\alpha := \|BS(\mu_n,\sigma_n)-\md\|_{2},\qquad
{\rm err}_n^\alpha = \|(\mu_n,\sigma_n)-(\mu_\alpha,\sigma_\alpha)\|_{H^1}
$$
are depicted in Figure~\ref{fig:exp}. For comparison, we also display the theoretical convergence curve. 

\begin{figure}[ht]
  \includegraphics[width=0.4\textwidth]{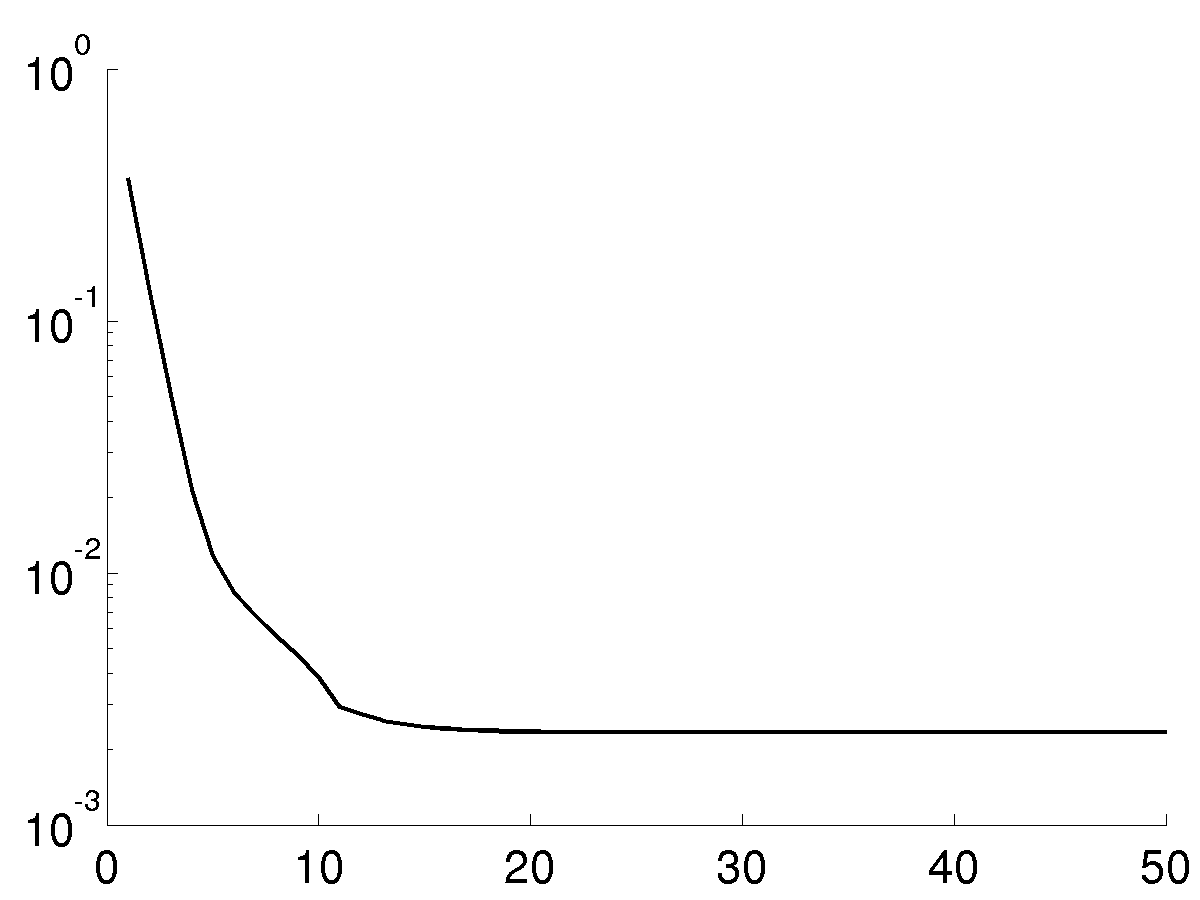}\quad
  \includegraphics[width=0.4\textwidth]{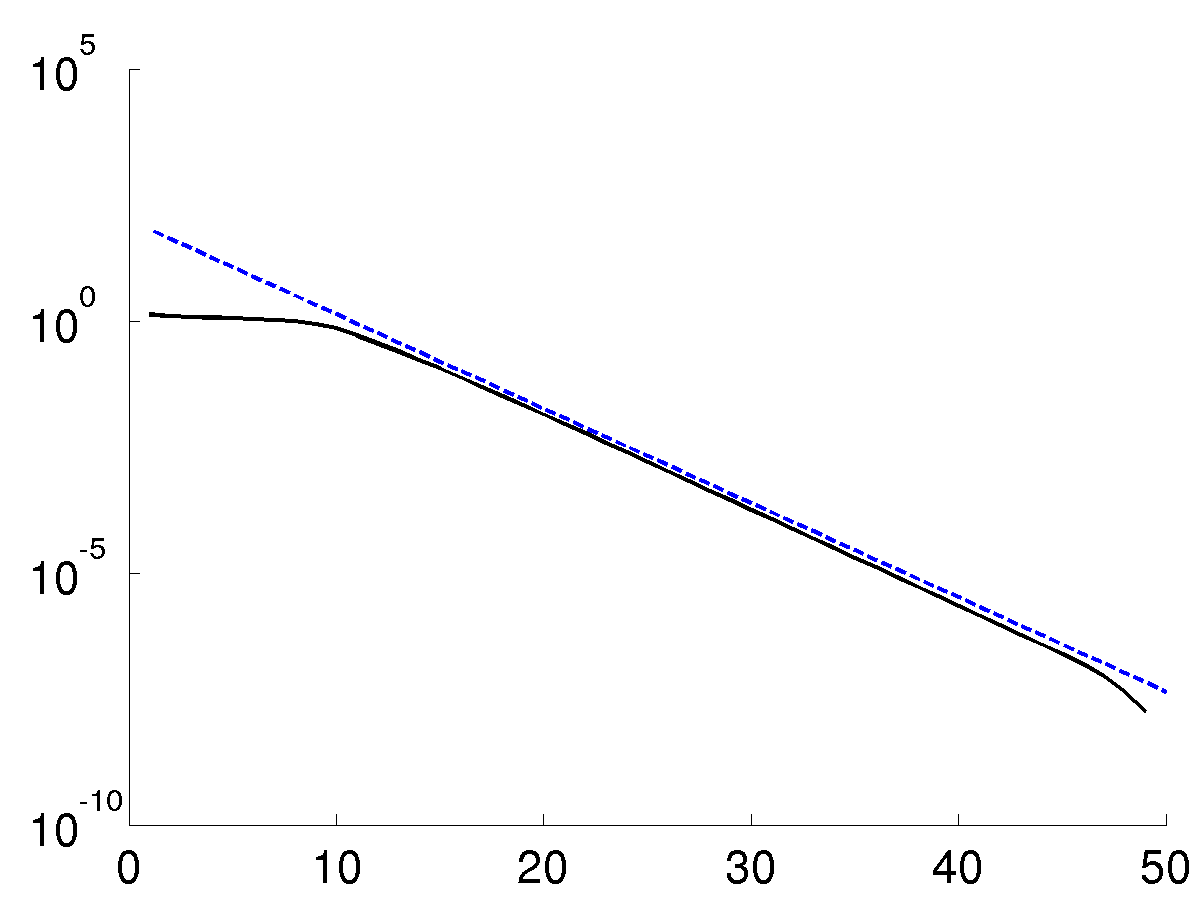}
  \caption{\label{fig:exp} Convergence of the PGN method for fixed $\alpha=10^{-5}$. Left: residual ${\rm res}_n^\alpha$. Right: linear convergence of ${\rm err}_n^\alpha$ and $(0.65)^n$ (dotted).}
  \end{figure}

In the first iterations, $\alpha_n$ is still rather large and the iterates stay within the vicinity of the initial guess. 
After $\alpha_n$ decreased sufficiently, the convergence of the error ${\rm err}_n^\alpha$ gets linear, i.e. ${\rm err}_n^\alpha \leq C \rho^n$ for some $0<\rho<1$. The residuals do not converge to zero here, since the minimizer 
$(\mu_\alpha,\sigma_\alpha)$ does not solve the inverse problem \eqref{eq:ip_informal} exactly. 
The residuals and the errors are however monotonically decreasing, which highlights the stability of the method.

\section{Conclusions}
In this paper we investigated numerical methods for reconstructing scattering and absorption rates in stationary radiative transfer from boundary observations. 
For a stable solution of this inverse problem, we considered Tikhonov regularization which leads to an optimal control problem constrained by an integro partial differential equation.
Using some sort of compactness provided by the averaging lemma, we were able to prove the weak continuity of the parameter-to-solution mapping. This allows us to show existence and stability of minimizers. We also established important differentiability properties which are required for the convergence of iterative minimization algorithms.
We discussed the convergence of a projected Gau\ss-Newton method.
Under the typical source condition, which is also required for nonlinear regularization theory, we could establish local convexity of the Tikhonov functional in the vicinity of minimizers, and thus obtained local linear convergence of the projected Gau\ss-Newton method.
It would be interesting to know, if convergence of iterative minimization algorithms can be shown without some sort of source condition.


\begin{thebibliography}{10}
\providecommand{\url}[1]{{#1}}
\providecommand{\urlprefix}{URL }
\expandafter\ifx\csname urlstyle\endcsname\relax
  \providecommand{\doi}[1]{DOI~\discretionary{}{}{}#1}\else
  \providecommand{\doi}{DOI~\discretionary{}{}{}\begingroup
  \urlstyle{rm}\Url}\fi

\bibitem{AcarVogel}
Acar, R., Vogel, C.R.: Analysis of bounded variation penalty methods for
  ill-posed problems.
\newblock Inverse Problems \textbf{10}, 1217--1229 (1994)

\bibitem{Arridge99}
Arridge, S.R.: Optical tomography in medical imaging.
\newblock Inverse Problems \textbf{15}(2), R41--R93 (1999)

\bibitem{BakKok04}
Bakushinsky, A.B., Kokurin, M.Y.: Iterative Methods for Approximate Solution of
  Inverse Problems, \emph{Mathematics and its Applications}, vol. 577.
\newblock Springer, Dordrecht (2004)

\bibitem{Bal08}
Bal, G.: Inverse transport from angularly averaged measurements and time
  harmonic isotropic sources.
\newblock In: A.L. Y.~Censor M.~Jiang (ed.) Mathematical Methods in Biomedical
  Imaging and Intesity-Modulated Radiation Therapy, CRM, pp. 19--35. Scuola
  Normale Superiore Pisa, Italy (2008)

\bibitem{BalJol08}
Bal, G., Jollivet, A.: Stability estimates in stationary inverse transport.
\newblock Inverse Probl. Imaging \textbf{2}(4), 427--454 (2008)

\bibitem{CaseZweifel67}
Case, K.M., Zweifel, P.F.: Linear transport theory.
\newblock Addison-Wesley Publishing Co., Reading (1967)

\bibitem{Cercignani:1988}
Cercignani, C.: The {B}oltzmann Equation and Its Applications.
\newblock Springer-Verlag, Berlin (1988)

\bibitem{Chandrasekhar60}
Chandrasekhar, S.: Radiative Transfer.
\newblock Dover Publications, Inc. (1960)

\bibitem{ChoulliStefanov98}
Choulli, M., Stefanov, P.: An inverse boundary value problem for the stationary
  transport equation.
\newblock Osaka J. Math. \textbf{36}(1), 87--104 (1998)

\bibitem{DL93vol6}
Dautray, R., Lions, J.L.: Mathematical Analysis and Numerical Methods for
  Science and Technology, Evolution Problems {II}, vol.~6.
\newblock Springer, Berlin (1993)

\bibitem{DiDoNaPaSi02}
Dierkes, T., Dorn, O., Natterer, F., Palamodov, V., Sielschott, H.: Fr\'echet
  derivatives for some bilinear inverse problems.
\newblock SIAM J. Appl. Math. \textbf{62}(6), 2092--2113 (2002)

\bibitem{Dorn98}
Dorn, O.: A transport-backtransport method for optical tomography.
\newblock Inverse Problems \textbf{14}, 1107--1130 (1998)

\bibitem{EggSch11}
Egger, H., Schlottbom, M.: Efficient reliable image reconstruction schemes for
  diffuse optical tomography.
\newblock Inv. Probl. Sci. Engrg. \textbf{19}, 155--180 (2011)

\bibitem{EggSch10:3}
Egger, H., Schlottbom, M.: A mixed variational framework for the radiative
  transfer equation.
\newblock Mathematical Models and Methods in Applied Sciences \textbf{03}(22),
  1150,014 (2012).
\newblock \doi{10.1142/S021820251150014X}.
\newblock \urlprefix\url{http://dx.doi.org/10.1142/S021820251150014X}

\bibitem{EggSch13LP}
Egger, H., Schlottbom, M.: An ${L}^p$ theory for stationary radiative transfer.
\newblock Applicable Analysis  (2013).
\newblock \doi{10.1080/00036811.2013.826798}

\bibitem{EHN96}
Engl, H.W., Hanke, M., Neubauer, A.: Regularization of inverse problems,
  \emph{Mathematics and its Applications}, vol. 375.
\newblock Kluwer Academic Publishers Group, Dordrecht (1996)

\bibitem{EnKuNeu89}
Engl, H.W., Kunisch, K., Neubauer, A.: Convergence rates for {T}ikhonov
  regularization of nonlinear ill-posed problems.
\newblock Inverse Problems \textbf{5}, 523--540 (1989)

\bibitem{GoLiPeSe88}
Golse, F., Lions, P.L., Perthame, B., Sentis, R.: Regularity of the moments of
  the solution of a transport equation.
\newblock Journal of Functional Analysis \textbf{76}(1), 110--125 (1988).
\newblock \doi{10.1016/0022-1236(88)90051-1}

\bibitem{HiPiUlUl09}
Hinze, M., Pinnau, R., Ulbrich, M., Ulbrich, S.: Optimization with {PDE}
  Constraints, \emph{Mathematical Modelling: Theory and Applications}, vol.~23.
\newblock Springer Science + Business Media B.V. (2009)

\bibitem{HoKaPoSch07}
Hofmann, B., Kaltenbacher, B., P\"oschl, C., Scherzer, O.: A convergence rates
  result for tikhonov regularization in banach spaces with non-smooth
  operators.
\newblock Inv. Prob. \textbf{23}, 987--1010 (2007)

\bibitem{KaltenbacherHofmann2010}
Kaltenbacher, B., Hofmann, B.: Convergence rates for the iteratively
  regularized gauss–newton method in banach spaces.
\newblock Inverse Problems \textbf{26}(3), 035,007 (2010).
\newblock \urlprefix\url{http://stacks.iop.org/0266-5611/26/i=3/a=035007}

\bibitem{KalNeu06}
Kaltenbacher, B., Neubauer, A.: Convergence of projected iterative
  regularization methods for nonlinear problems with smooth solutions.
\newblock Inverse Problems \textbf{22}, 1105--1119 (2006)

\bibitem{KaNeuSche08}
Kaltenbacher, B., Neubauer, A., Scherzer, O.: Iterative Regularized Methods for
  Nonlinear Ill-Posed Problems, \emph{Radon Series on Computational and Applied
  Mathematics}, vol.~6.
\newblock Walter de Gruyter (2008)

\bibitem{KalSchSch09}
Kaltenbacher, B., Sch\"opfer, F., Schuster, T.: Iterative methods for nonlinear
  ill-posed problems in {B}anach spaces: convergence and applications to
  parameter identification problems.
\newblock Inverse Problems \textbf{25}, 065,003 (19pp) (2009)

\bibitem{KloHie99}
Klose, A.D., Hielscher, A.H.: Iterative reconstruction scheme for optical
  tomography based on the equation of radiative transfer.
\newblock Med. Phys. \textbf{26}(8), 1698--1707 (1999)

\bibitem{Kondratyev69}
Kondratyev, K.Y.: Radiation in the Atmosphere.
\newblock Academic Press (1969)

\bibitem{LewisMiller84}
Lewis, E.E., Miller~Jr., W.F.: Computational Methods of Neutron Transport.
\newblock John Wiley \& Sons, Inc., New York Chichester Brisbane Toronto
  Singapore (1984)

\bibitem{McDowallStefanovTamason2010}
McDowall, S., Stefanov, P., Tamasan, A.: Stability of the gauge equivalent
  classes in inverse stationary transport.
\newblock Inverse Problems \textbf{26}(2), 025,006 (2010).
\newblock \urlprefix\url{http://stacks.iop.org/0266-5611/26/i=2/a=025006}

\bibitem{Peraiah04}
Peraiah, A.: An Introduction to Radiative Transfer -- Methods and applications
  in astrophysics.
\newblock Cambridge University Press (2004)

\bibitem{Ram03}
Ramlau, R.: {TIGRA} -- an iterative algorithm for regularizing nonlinear
  ill-posed problems.
\newblock Inverse Problems \textbf{19}, 433--465 (2003)

\bibitem{Resmerita2005}
Resmerita, E.: Regularization of ill-posed problems in banach spaces:
  convergence rates.
\newblock Inverse Problems \textbf{21}(4), 1303 (2005).
\newblock \urlprefix\url{http://stacks.iop.org/0266-5611/21/i=4/a=007}

\bibitem{ScherzerVariationalMethodsInImaging2009}
Scherzer, O., Grasmair, M., Grossauer, H., Haltmeier, M., Lenzen, F.:
  Variational Methods in Imaging.
\newblock Springer (2009)

\bibitem{ThesisSch}
Schlottbom, M.: On forward and inverse models in optical tomography.
\newblock Ph.D. thesis, RWTH Aachen (2011).
\newblock
  \urlprefix\url{http://darwin.bth.rwth-aachen.de/opus3/volltexte/2011/3857/}

\bibitem{SchusterKaltenbacherHofmannKazimierski2012}
Schuster, T., Kaltenbacher, B., Hofmann, B., Kazimierski, K.S.: Regularization
  Methods in Banach Spaces.
\newblock De Gruyter (2012)

\bibitem{TangHanHan2013}
Tang, J., Han, W., Han, B.: A theoretical study for {RTE}-based parameter
  identification problems.
\newblock Inverse Problems \textbf{29}(9), 095,002 (2013).
\newblock \urlprefix\url{http://stacks.iop.org/0266-5611/29/i=9/a=095002}

\bibitem{WrightArridgeSchweiger07}
Wright, S., Schweiger, M., Arridge, S.R.: Reconstruction in optical tomography
  using {$P_N$} approximations.
\newblock Meas. Sci. Technol. \textbf{18}, 79--86 (2007)

\end{thebibliography}

\end{document}